\documentclass[11pt,leqno]{article}
\usepackage{a4wide}
\usepackage{amssymb,upgreek}
\usepackage{amsmath}
\usepackage{amsthm}
\usepackage{mathrsfs}
\usepackage{euscript,csquotes}
\usepackage[curve]{xypic}
\usepackage[usenames,dvipsnames]{xcolor}
\usepackage{endnotes,comment}

\usepackage{hyperref}

\theoremstyle{plain}

\newcommand{\con}{\operatorname{con}}
\newcommand{\Hom}{\operatorname{Hom}}
\newcommand{\End}{\operatorname{End}}
\newcommand{\Ext}{\operatorname{Ext}}

\newcommand{\Mod}{\operatorname{Mod}}

\newcommand{\ind}{\operatorname{ind}}

\newcommand{\res}{\operatorname{res}}
\newcommand{\cores}{\operatorname{cores}}

\newcommand{\Vect}{\operatorname{Vec}}

\newtheorem{theorem}{Theorem}[section]
\newtheorem{corollary}[theorem]{Corollary}
\newtheorem{lemma}[theorem]{Lemma}
\newtheorem{proposition}[theorem]{Proposition}

\newtheorem{definition}[theorem]{Definition}

\theoremstyle{definition}
\newtheorem{remark}[theorem]{Remark}
\title{\textbf{Duals in natural characteristic}}
\author{Peter Schneider, Claus Sorensen}

\date{}

\begin{document}
\maketitle

\begin{abstract}
In this article we introduce a derived smooth duality functor $R\underline{\Hom}(-,k)$ on the unbounded derived category $D(G)$ of smooth $k$-representations of a $p$-adic Lie group $G$. Here 
$k$ is a field of characteristic $p$. Using this functor we relate various subcategories of admissible complexes in $D(G)$. 
\end{abstract}

\section{Introduction}

Let $G$ be a $p$-adic Lie group of dimension $d$, and let $k$ be a field of characteristic $p$. We denote by $\Mod(G)$ the abelian category of smooth $G$-representations in $k$-vector spaces.


In this paper we endow the unbounded derived category $D(G)=D(\Mod(G))$ with a tensor product $\otimes_k$ plus internal homs $R\underline{\Hom}$, and 
begin exploring the resulting closed symmetric monoidal category. The duality functor $R\underline{\Hom}(-,k)$ is of particular interest to us. It gives a derived approach to the higher smooth duality functors $S^j$ introduced by Kohlhaase in \cite{Koh}.

Our first result (Proposition \ref{prop:SSduality}) shows that the functors $S^j$ are compatible with duals on the Hecke side. If $H_U$ denotes the Hecke algebra of a torsion free
open pro-$p$ subgroup $U \subseteq G$, we give an $H_U$-equivariant spectral sequence with $E_2$-page $H^i(U,S^j(V))$ converging to the twisted dual Hecke modules
$H^{d-i-j}(U,V)^{\vee}(\chi_G)$. Here the character $\chi_G: G \rightarrow k^{\times}$ turns out to coincide with the duality character in \cite{Koh}. This is a non-trivial fact and we give a proof. In particular $\chi_G=1$ if $G$ is an open subgroup of the $\frak{F}$-points of a connected reductive group over a $p$-adic field $\frak{F}$.

Motivated by \cite{DGA}, which gives a differential graded version of the Hecke algebra $H_U^{\bullet}$ along with an equivalence between $D(G)$ and the derived category $D(H_U^{\bullet})$ of differential graded modules over $H_U^{\bullet}$, we turn to studying the functor $R\underline{\Hom}(-,k)$ in the derived setting. 

We first observe that $R\underline{\Hom}(-,k)$ is involutive on the subcategory $D_{adm}(G)$ of complexes $V^{\bullet}$ with admissible cohomology representations 
$h^i(V^{\bullet})$ for all $i\in \Bbb{Z}$. We then introduce a possibly larger subcategory 
$$
D(G)^a \supseteq D_{adm}(G)
$$
consisting of globally admissible complexes, by which we mean $H^i(U,V^{\bullet})$ is finite-dimensional for all $i\in \Bbb{Z}$. As we show in Proposition \ref{prop:involutive2}, a complex $V^{\bullet}$ belongs to $D(G)^a$ precisely when the natural 
biduality morphism 
$$
\eta_{V^{\bullet}}: V^{\bullet} \longrightarrow R\underline{\Hom}(R\underline{\Hom}(V^{\bullet},k),k)
$$
is a quasi-isomorphism. As a result, the notion of being globally admissible is independent of the choice of $U$. Finally we show that a globally admissible $V^{\bullet}$ satisfying various boundedness conditions actually lies in the subcategory $D_{adm}(G)$. For instance, Corollary \ref{coro:finitedim2} tells us $D_{adm}^b(G)$ contains exactly those  complexes $V^{\bullet}$ whose total cohomology $H^*(U,V^{\bullet})$ is finite-dimensional. 

To orient the reader we point out that $D(G)^a$ is equivalent to the category $D_{fin}(H_U^{\bullet})$ of differential graded $H_U^{\bullet}$-modules with finite-dimensional cohomology spaces in each degree. We have work in progress aiming at an intrinsic description of the duality functor on $D(H_U^{\bullet})$ corresponding to $R\underline{\Hom}(-,k)$.


\section{Higher smooth duality}

For any compact open subgroup $K \subseteq G$ we have the completed group ring $\Omega(K)$ of $K$ over $k$. This is a noetherian ring (cf.\ \cite{pLG} Thm.\ 33.4). We let $\Mod(\Omega(K))$ denote the abelian category of left $\Omega(K)$-modules. But $\Omega(K)$ also is a pseudocompact ring (cf.\ \cite{pLG} IV\S19). We therefore also have the abelian category $\Mod_{pc}(\Omega(K))$ of pseudocompact left $\Omega(K)$-modules together with the obvious forgetful functor $\Mod_{pc}(\Omega(K)) \rightarrow \Mod(\Omega(K))$. Both categories have enough projective objects. Any finitely generated $\Omega(K)$-module $M$ is pseudocompact in a natural way. This leads to the natural isomorphism
\begin{equation}\label{f:Ext-iso}
  \Ext_{\Mod_{pc}(\Omega(K))}^*(M,N) \cong \Ext_{\Mod(\Omega(K))}^*(M,N)
\end{equation}
for any finitely generated module $M$ in $\Mod(\Omega(K))$ and any pseudocompact module $N$ in $\Mod_{pc}(\Omega(K))$.

Pontrjagin duality gives rise to the equivalence of categories
\begin{align*}
  \Mod(K)^{op} & \xrightarrow{\;\simeq\;} \Mod_{pc}(\Omega(K)) \\
        V & \longmapsto V^\vee := \Hom_k(V,k)
\end{align*}
where, of course, in order to make $V^\vee$ a left module we use the inversion map $g \mapsto g^{-1}$ on $K$. In particular, we have the natural isomorphisms
\begin{equation*}
  \Ext_{\Mod(K)}^*(V_1,V_2) \cong \Ext_{\Mod_{pc}(\Omega(K))}^*(V_2^\vee,V_1^\vee) \ .
\end{equation*}
If we apply this with the trivial $K$-representation $V_2 := k$ and use \eqref{f:Ext-iso} we obtain the natural isomorphism
\begin{equation}\label{f:K-Kohlhaase}
  \Ext_{\Mod(K)}^*(V,k) \cong \Ext_{\Mod(\Omega(K))}^*(k,V^\vee)
\end{equation}
for $V$ in $\Mod(K)$.

If $K' \subseteq K$ is another open subgroup then in \eqref{f:K-Kohlhaase} we have on both sides the obvious restriction maps. Hence we may pass to the inductive limit
\begin{equation}\label{f:Kohlhaase}
  \varinjlim_K \Ext_{\Mod(K)}^*(V,k) \cong \varinjlim_K \Ext_{\Mod(\Omega(K))}^*(k,V^\vee) \ .
\end{equation}
Note that, for $V$ in $\Mod(G)$, the right hand side are Kohlhaase's higher smooth dual functors
\begin{equation*}
  S^*(V) := \varinjlim_K \Ext_{\Mod(\Omega(K))}^*(k,V^\vee)
\end{equation*}
in \cite{Koh}. We use the left hand side to understand these as derived functors. For any $V_1, V_2$ in $\Mod(G)$ we introduce
\begin{align*}
  \underline{\Hom}(V_1,V_2) & := \{f \in \Hom_k(V_1,V_2) : f \ \text{is $K$-equivariant} \\
   & \qquad\qquad\qquad\qquad \text{for some compact open subgroup $K \subseteq G$} \}.
\end{align*}
Via the $G$-action defined by ${^g f} := g f (g^{-1} -)$, for $g \in G$, this is again an object in $\Mod(G)$. Since the functors
\begin{align*}
  \Mod(G) & \longrightarrow \Mod(G) \\
      V_2 & \longmapsto \underline{\Hom}(V_1, V_2)
\end{align*}
are left exact we have the corresponding right derived functors
\begin{equation*}
  \underline{\Ext}^i(V_1, V_2)   \qquad\text{for $i \geq 0$}.
\end{equation*}

\begin{lemma}\phantomsection\label{lemma:Ext}
\begin{itemize}
   \item[i.] If $V_2$ is injective in $\Mod(G)$ then $\underline{\Hom}(V_1,V_2)$ is $H^0(U,-)$-acyclic for any compact open subgroup $U \subseteq G$.
   \item[ii.] $\underline{\Ext}^*(V_1,V_2) = \varinjlim_K \Ext_{\Mod(K)}^*(V_1,V_2)$.
\end{itemize}
\end{lemma}
\begin{proof}
Note that any injective object in $\Mod(G)$ remains injective when viewed in $\Mod(U)$. Therefore this is Prop.\ 2.2 in the appendix by Verdier in \cite{CG}.
\end{proof}

We see that, in particular, we can rewrite Kohlhaase's functors as the derived functors
\begin{equation*}
  S^*(V) = \underline{\Ext}^*(V,k) \ .
\end{equation*}

\begin{remark}\label{rem:global-dim}
By \cite{Bru} Thm.\ 4.1 the global dimension of $\Omega(K)$ as a pseudocompact ring is equal to the cohomological dimension of $K$. By Lazard (cf.\ \cite{CG} I-47) the latter is equal to $d$ provided $K$ is pro-$p$ and torsion free. Since $G$ contains arbitrarily small open pro-$p$ subgroups without torsion we conclude from Lemma \ref{lemma:Ext}.ii  that $\underline{\Ext}^i(V_1, V_2) = 0$ for any $i > d$.
\end{remark}

\begin{proposition}\label{prop:SS}
   For any compact open subgroup $U \subseteq G$ we have the $E_2$-spectral sequence
\begin{equation*}
  H^i(U,\underline{\Ext}^j(V_1,V_2)) \Longrightarrow \Ext_{\Mod(U)}^{i+j}(V_1,V_2) \ .
\end{equation*}
In particular,
\begin{equation*}
  H^i(U,S^j(V)) \Longrightarrow \Ext_{\Mod(U)}^{i+j}(V,k) \ .
\end{equation*}
\end{proposition}
\begin{proof}
This is the composed functor spectral sequence which exists by Lemma \ref{lemma:Ext}.i.
\end{proof}

The above spectral sequence has an additional equivariance property which we now describe. We fix a compact open subgroup $U \subseteq G$ and consider the compact induction $\mathbf{X}_U := \ind_U^G(k)$ in $\Mod(G)$. We then have the endomorphism ring $H_U := \End_{\Mod(G)}(\mathbf{X}_U)^{op}$ so that $\mathbf{X}_U$ becomes a right $H_U$-module. Frobenius reciprocity gives a natural isomorphism of functors $H^0(U,-) \cong \Hom_{\Mod(G)}(\mathbf{X}_U,-)$ on $\Mod(G)$. By using injective resolutions it extends to a natural isomorphism of cohomological functors
\begin{equation*}
  H^*(U,-) \cong \Ext_{\Mod(G)}^*(\mathbf{X}_U,-) \ .
\end{equation*}
Through its right action on $\mathbf{X}_U$ the right hand side becomes a left $H_U$-module. In this way $H^*(U,-)$ is equipped with a left $H_U$-module structure. In particular, $\Hom_{\Mod(U)}(V_1,V_2) = H^0(U, \underline{\Hom}(V_1,V_2)) \cong \Hom_{\Mod(G)}(\mathbf{X}_U,\underline{\Hom}(V_1,V_2))$ carries a left $H_U$-module structure which is functorial in $V_1$ and $V_2$. By derivation we obtain a functorial left $H_U$-module structure on $\Ext_{\Mod(U)}^*(V_1,V_2)$.

\begin{remark}\label{rem:equiv}
   The spectral sequence in Prop.\ \ref{prop:SS} is $H_U$-equivariant.
\end{remark}
\begin{proof}
This is straightforward from the way the composed functor spectral sequence is constructed.
\end{proof}

We now suppose in addition that $U$ is pro-$p$ and torsion free. Then $U$ is a Poincar\'e group of dimension $d$ (\cite{CG} I-47 Ex.\ (3)). A straightforward variant of the appendix by Verdier in \cite{CG} therefore gives the following: In $\Mod(U)$ we have the dualizing object
\begin{equation*}
  \hat{I} := \varinjlim_{K \subseteq U, \cores} \Hom_k(H^d(K,k),k) \ ,
\end{equation*}
which actually is isomorphic to the trivial representation $k$ in $\Mod(U)$, together with an isomorphism
\begin{equation}\label{f:duality-iso}
  \Hom_k(H^i(U,V),k) \cong \Ext^{d-i}_{\Mod(U)}(V,\hat{I}) \cong \Ext^{d-i}_{\Mod(U)}(V,k)   \qquad\text{for any $i \geq 0$}
\end{equation}
which is natural in $V$ in $\Mod(U)$; this latter isomorphism is induced by the Yoneda product
\begin{equation*}
  \Ext^{d-i}_{\Mod(U)}(V,\hat{I}) \times H^i(U,V) \longrightarrow H^d(U,\hat{I}) \cong H^d(U,k) (\cong k)
\end{equation*}
(Def.\ 4.5, Prop.\ 3.1.5, and first displayed formula on p.\ V-20). In the following we will keep writing $\hat{I}$ (instead of $k$) but will view it even as a trivial $G$-representation. From now on we assume that $V$ lies in $\Mod(G)$ and we will see that then all terms in the above Yoneda pairing carry a natural left $H_U$-action.
\begin{itemize}
  \item[A.]  From the proof of Prop.\ 8.4.i in \cite{OS} we know a formula for the $H_U$ action on $H^*(U,V)$. Viewing $H_U$ as the convolution algebra of $U$-bi-invariant functions with compact support on $G$ we denote by $\tau_h \in H_U$, for $h \in G$, the characteristic function of the double coset $UhU$ in $G$. The diagram
\begin{equation*}
  \xymatrix{
    H^*(U,V) \ar[d]_{\res} \ar[r]^{\tau_h \cdot} & H^*(U,V)  \\
    H^*(U \cap h^{-1}Uh,V) \ar[r]^{h_\ast} & H^*(U \cap hUh^{-1},V) \ar[u]_{\cores}  }
\end{equation*}
is commutative.
  \item[B.] By \cite{CG} I Prop.\ 18 the same $\hat{I}$ is also a dualizing object in $\Mod(U')$ for any open subgroup $U' \subseteq U$.
  \item[C.] As introduced above, we have a natural left $H_U$-action on $\Ext^*_{\Mod(U)}(V,\hat{I})$. To give an explicit formula we let $V'$ be any other object in $\Mod(G)$ and we first recall that, for any open subgroup $U' \subseteq U$ and any $h \in G$, we have the following natural maps:
      \begin{itemize}
        \item[--] The restriction map $\Ext^*_{\Mod(U)}(V,V') \xrightarrow{\res} \Ext^*_{\Mod(U')}(V,V')$ which derives the obvious forgetful map on homomorphisms.
        \item[--] The corestriction map $\Ext^*_{\Mod(U')}(V,V') \xrightarrow{\cores} \Ext^*_{\Mod(U)}(V,V')$ which derives the map which sends a $U'$-equivariant homomorphism $f : V \rightarrow V'$ to the $U$-equivariant homomorphism $\sum_{g \in U/U'} g f (g^{-1}-) : V \rightarrow V'$.
        \item[--] The conjugation map $\Ext^*_{\Mod(U)}(V,V') \xrightarrow{h_*} \Ext^*_{\Mod(hUh^{-1})}(V,V')$ which derives the map which sends a $U$-equivariant homomorphism $f : V \rightarrow V'$ to the $hUh^{-1}$-equivariant homomorphism $h f(h^{-1} -) : V \rightarrow V'$.
      \end{itemize}
      As for A. it is straightforward to verify that, for any $h \in G$, the diagram
\begin{equation*}
  \xymatrix{
    \Ext^*_{\Mod(U)}(V,V') \ar[d]_{\res} \ar[r]^{\tau_h \cdot} & \Ext^*_{\Mod(U)}(V,V')  \\
    \Ext^*_{\Mod(U \cap h^{-1}Uh)}(V,V') \ar[r]^{h_\ast} & \Ext^*_{\Mod(U \cap hUh^{-1})}(V,V') \ar[u]_{\cores}  }
\end{equation*}
     is commutative.
  \item[E.] It is easily checked that the map
\begin{align*}
  H_U & \longrightarrow H_U \\
  \tau & \longmapsto \tau(-^{-1})
\end{align*}
is an anti-involution of the $k$-algebra $H_U$. It sends $\tau_h$ to $\tau_{h^{-1}}$.
\end{itemize}

\begin{lemma}\label{lemma:pairings}
   For any $0 \leq i \leq d$ and any $h \in G$ the diagram of Yoneda pairings
\begin{equation*}
  \xymatrix{
    \Ext^{d-i}_{\Mod(U)}(V,\hat{I}) \ar[d]_{\res}  & \times  & H^i(U,V) \ar[r] & H^d(U,\hat{I})  \\
    \Ext^{d-i}_{\Mod(U \cap h^{-1}Uh)}(V,\hat{I}) \ar[d]_{h_*}  & \times  & H^i(U \cap h^{-1}Uh,V) \ar[u]_{\cores} \ar[r] & H^d(U \cap h^{-1}Uh,\hat{I}) \ar[u]_{\cores} \ar[d]^{h_*} \\
    \Ext^{d-i}_{\Mod(U \cap hUh^{-1})}(V,\hat{I}) \ar[d]_{\cores}  & \times  & H^i(U \cap hUh^{-1},V) \ar[u]_{h_*^{-1}} \ar[r] & H^d(U \cap hUh^{-1},\hat{I}) \ar[d]^{\cores} \\
    \Ext^{d-i}_{\Mod(U)}(V,\hat{I})  & \times  & H^i(U,V) \ar[u]_{\res} \ar[r] & H^d(U,\hat{I})   }
\end{equation*}
   is commutative.
\end{lemma}
\begin{proof} We fix injective resolutions $V \xrightarrow{\simeq} \mathcal{J}^\bullet$ and $\hat{I} \xrightarrow{\simeq} \mathcal{I}^\bullet$ in $\Mod(G)$, which remain injective resolutions after restriction to any given open subgroup of $G$.

\textit{The upper rectangle:} Let $\beta^\bullet : \mathcal{J}^\bullet \rightarrow \mathcal{I}^\bullet[d-i]$ a $U$-equivariant and $\alpha^\bullet : \mathcal{I}^\bullet \rightarrow \mathcal{J}^\bullet[i]$ a $U \cap h^{-1}Uh$-equivariant homomorphism of complexes representing classes $[\beta^\bullet] \in \Ext^{d-i}_{\Mod(U)}(V,\hat{I})$ and $[\alpha^\bullet] \in H^i(U,V)$, respectively. Then $\beta^\bullet$ also represents $\res [\beta^\bullet]$ whereas $\cores [\alpha^\bullet]$ is represented by $\sum_{g \in U / U \cap h^{-1}Uh} {^g \alpha}^\bullet$. We compute
\begin{align*}
  [\beta^\bullet[i]] \circ \cores [\alpha^\bullet] & = [\beta^\bullet[i] \circ \sum_{g \in U / U \cap h^{-1}Uh} {^g \alpha}^\bullet] = [\beta^\bullet[i] \circ \sum_{g \in U / U \cap h^{-1}Uh} g \alpha^\bullet(g^{-1} -)]   \\
     & = [\sum_{g \in U / U \cap h^{-1}Uh} g(\beta^\bullet[i] \circ \alpha^\bullet)(g^{-1} -)] = [\sum_{g \in U / U \cap h^{-1}Uh} {^g (}\beta^\bullet[i] \circ \alpha^\bullet)]  \\
     & = \cores (\res[\beta^\bullet[i]] \circ [\alpha^\bullet]) \ .
\end{align*}

\textit{The middle rectangle:} Let $\beta^\bullet : \mathcal{J}^\bullet \rightarrow \mathcal{I}^\bullet[d-i]$ a $U \cap h^{-1}Uh$-equivariant and $\alpha^\bullet : \mathcal{I}^\bullet \rightarrow \mathcal{J}^\bullet[i]$ a $U \cap hUh^{-1}$-equivariant homomorphism of complexes representing classes $[\beta^\bullet] \in \Ext^{d-i}_{\Mod(U) \cap h^{-1}Uh}(V,\hat{I})$ and $[\alpha^\bullet] \in H^i(U \cap hUh^{-1},V)$, respectively. The $h_*[\beta^\bullet]$ and $h_*^{-1}[\alpha^\bullet]$ are represented by ${^h \beta}^\bullet$ and ${^{h^{-1}} \alpha}^\bullet$. We compute
\begin{align*}
  h_*[\beta^\bullet[i]] \circ [\alpha^\bullet] & = [{^h \beta}^\bullet[i] \circ \alpha^\bullet] = [h\beta^\bullet[i](h^{-1}\alpha^\bullet(-))] = [h\beta^\bullet[i](h^{-1}\alpha^\bullet(h h^{-1}-))]  \\
  & = [h\beta^\bullet[i]({^{h^{-1}} \alpha^\bullet}(h^{-1}-))] = [h (\beta^\bullet[i] \circ {^{h^{-1}} \alpha^\bullet})(h^{-1}-)] = [{^h (}\beta^\bullet[i] \circ {^{h^{-1}} \alpha^\bullet})]   \\
  & = h_* ([\beta^\bullet[i]] \circ h_*^{-1}[\alpha^\bullet]) \ .
\end{align*}

\textit{The lower rectangle:} This is entirely analogous to the computation for the upper rectangle.
\end{proof}

At this point we fix an isomorphism $H^d(U,\hat{I}) \cong k$ and henceforth treat it as an identification. Since, for any other open pro-$p$ and torsion free subgroup $U' \subseteq G$, the corestriction maps $H^d(U',\hat{I}) \xleftarrow{\cong} H^d(U' \cap U,\hat{I}) \xrightarrow{\cong} H^d(U,\hat{I})$ are isomorphisms (cf.\ \cite{CG} I-50(4)), this induces a corresponding identification $H^d(U',\hat{I}) = k$. In particular, under these identifications the conjugation isomorphism $h_* : H^d(U \cap h^{-1}Uh,\hat{I}) \xrightarrow{\cong} H^d(U \cap hUh^{-1},\hat{I})$, for any $h \in G$, becomes the multiplication by a scalar $\chi_G(h) \in k^\times$.

\begin{lemma}\label{lemma:chiG}
   The map $\chi_G : G \rightarrow k^\times$ is a character which is independent of $U$ and is trivial on any pro-$p$ subgroup of $G$.
\end{lemma}
\begin{proof}
The independence from the chosen identification $H^d(U,\hat{I}) = k$ as well as he triviality on $U$ are obvious. The former implies the independence from $U$ and hence, by the latter, the triviality on any open torsion free pro-$p$ subgroup of $G$. Suppose that we have checked the multiplicativity of $\chi_G$ already and let $U_0$ be any pro-$p$ subgroup of $G$. Note that, as a $p$-adic Lie group, $G$ always has an open torsion free pro-$p$ subgroup. Hence $\chi_G | U_0$ factorizes through a finite quotient which is a $p$-group. Since any finite subgroup of $k^\times$ has order prime to $p$ it follows that $\chi_G$ is trivial on $U_0$. To establish multiplicativity let $g, h \in G$. Since conjugation commutes with corestriction we have the following three commutative diagrams, which together show our claim:
\begin{equation*}
  \xymatrix{
    H^d(U \cap gUg^{-1}, \hat{I})   & H^d(U \cap gUg^{-1} \cap ghU(gh)^{-1}, \hat{I}) \ar[l]_-{\cores}^-{\cong} \\
    H^d(U \cap g^{-1}Ug, \hat{I}) \ar[u]_{g_*}   & H^d(U \cap g^{-1}Ug \cap hUh^{-1}, \hat{I}) \ar[l]_-{\cores}^-{\cong} \ar[u]_{g_*} , }
\end{equation*}
\begin{equation*}
  \xymatrix{
    H^d(U \cap hUh^{-1}, \hat{I})   & H^d(U \cap g^{-1}Ug \cap hUh^{-1}, \hat{I}) \ar[l]_-{\cores}^-{\cong} \\
    H^d(U \cap h^{-1}Uh, \hat{I}) \ar[u]_{h_*}   & H^d(U \cap (gh)^{-1}Ugh \cap h^{-1}Uh, \hat{I}) \ar[l]_-{\cores}^-{\cong} \ar[u]_{h_*} , }
\end{equation*}
and
\begin{equation*}
  \xymatrix{
   H^d(U \cap gUg^{-1} \cap ghU(gh)^{-1}, \hat{I}) \ar[r]^-{\cores}_-{\cong} & H^d(U \cap ghU(gh)^{-1}, \hat{I})    \\
   H^d(U \cap g^{-1}Ug \cap hUh^{-1}, \hat{I}) \ar[r]^-{\cores}_-{\cong} \ar[u]_{g_*} & H^d(g^{-1}Ug \cap hUh^{-1}, \hat{I}) \ar[u]_{g_*}   \\
   H^d(U \cap (gh)^{-1}Ugh \cap h^{-1}Uh, \hat{I}) \ar[u]_{h_*} \ar[r]^-{\cores}_-{\cong} & H^d(U \cap (gh)^{-1}Ugh, \hat{I}) \ar[u]_{h_*}  \ar@/_15ex/[uu]_{(gh)_*}   . }
\end{equation*}
\end{proof}

The map
\begin{align*}
  H_U & \longrightarrow H_U \\
  \tau & \longmapsto \chi_G \tau \ \text{(pointwise product of functions)}
\end{align*}
is an algebra homomorphism. Pulling back an $H_U$-module $M$ along this homomorphism defines the twisted $H_U$-module $M(\chi_G)$. Also note that we may use the anti-involution in E. to make the $k$-linear dual $M^\vee := \Hom_k(M,k)$ of a left $H_U$-module $M$ again into a left $H_U$-module.

Using A. and C. we then may rewrite the diagram in Lemma \ref{lemma:pairings} as the commutative diagram
\begin{equation*}
  \xymatrix{
    \Ext^{d-i}_{\Mod(U)}(V,\hat{I}) \ar[d]_{\tau_h \cdot}  & \times  & H^i(U,V) \ar[r] & k \ar[d]^{\chi_G(h) \cdot} \\
    \Ext^{d-i}_{\Mod(U)}(V,\hat{I})  & \times  & H^i(U,V) \ar[u]_{\tau_{h^{-1}} \cdot} \ar[r] & k  . }
\end{equation*}
This says that the duality isomorphism \eqref{f:duality-iso} in fact is an isomorphism of $H_U$-modules
\begin{equation}\label{f:H-duality-iso}
  \Ext^{d-i}_{\Mod(U)}(V,k) \xrightarrow{\:\cong\;} H^i(U,V)^\vee(\chi_G) \ .
\end{equation}

\begin{proposition}\label{prop:SSduality}
  For any compact open subgroup $U \subseteq G$ which is pro-$p$ and torsion free and any $V$ in $\Mod(G)$ we have an $H_U$-equivariant $E_2$-spectral sequence
\begin{equation*}
   H^i(U,S^j(V)) \Longrightarrow H^{d-i-j}(U,V)^\vee(\chi_G) \ .
\end{equation*}
\end{proposition}
\begin{proof}
The spectral sequence arises by combining the second spectral sequence in Prop.\ \ref{prop:SS} (observe Remark \ref{rem:equiv}) with the duality isomorphism \eqref{f:H-duality-iso}.
\end{proof}

\begin{remark}
   Suppose that $G = \mathbf{G}(\mathfrak{F})$ where $\mathfrak{F}/\mathbb{Q}_p$ is a finite extension and $\mathbf{G}$ is a connected reductive $\mathfrak{F}$-split group over $\mathfrak{F}$. Assuming that a pro-$p$ Iwahori subgroup $U$ of $G$ is torsion free it is shown in \cite{OS} Prop.\ 7.16 that $\chi_G = 1$. Under additional assumptions this was proved before in \cite{Koz}.
\end{remark}

In fact, we will show that $\chi_G$ coincides with the duality character introduced by Kohlhaase in \cite{Koh} after Def.\ 3.12 and which we temporarily denote by $\chi_G^\mathrm{Koh}$.

\begin{proposition}\label{prop:Kohlhaase}
   We have $\chi_G = \chi_G^\mathrm{Koh}$.
\end{proposition}
\begin{proof}
The character $\chi_G^\mathrm{Koh}$ describes the $G$-action on a certain one dimensional $k$-vector space $E^d(k)$ the original definition of which we do not need. Instead we use \cite{Koh} Prop.\ 3.2 which says that, for any compact open subgroup $G_0 \subseteq G$, there is a natural $G_0$-equivariant isomorphism $\ell_{G,G_0} : E^d(k) \xrightarrow{\cong} \Ext^d_{\Mod(\Omega(G_0))}(k, \Omega(G_0))$ such that:
\begin{itemize}
  \item[1)] For any $g \in G$ the diagram
\begin{equation*}
  \xymatrix{
    E^d(k) \ar[d]_{\chi_G^\mathrm{Koh}(g) \cdot} \ar[rr]^-{\ell_{G,G_0}} && \Ext^d_{\Mod(\Omega(G_0))}(k, \Omega(G_0)) \ar[d]^{g_*} \\
    E^d(k) \ar[rr]^-{\ell_{G,g G_0 g^{-1}}} && \Ext^d_{\Mod(\Omega(g G_0 g^{-1}))}(k, \Omega(gG_0 g^{-1}))   }
\end{equation*}
is commutative, where $g_*$ is the conjugation isomorphism (compare the argument in the last paragraph of the proof of \cite{Koh} Prop.\ 3.13).
  \item[2)] For any open subgroup $G_1 \subseteq G_0$ the diagram
\begin{equation*}
  \xymatrix@R=0.5cm{
                &         \Ext^d_{\Mod(\Omega(G_0))}(k, \Omega(G_0)) \ar[dd]^{\ell_{G_0,G_1}}     \\
  E^d(k) \ar[ur]^-{\ell_{G,G_0}} \ar[dr]_-{\ell_{G,G_1}}                 \\
                &         \Ext^d_{\Mod(\Omega(G_1))}(k, \Omega(G_1))                 }
\end{equation*}
is commutative. Moreover $\ell_{G_0,G_1}$ is the composite of the restriction map
\begin{equation*}
  \Ext^d_{\Mod(\Omega(G_0))}(k, \Omega(G_0)) \xrightarrow{\res}  \Ext^d_{\Mod(\Omega(G_1))}(k, \Omega(G_0))
\end{equation*}
and the map
\begin{equation*}
  \Ext^d(k, j_{G_1.G_0}^\vee) : \Ext^d_{\Mod(\Omega(G_1))}(k, \Omega(G_0)) \rightarrow  \Ext^d_{\Mod(\Omega(G_1))}(k, \Omega(G_1))
\end{equation*}
which is induced by the Pontrjagin dual $j_{G_1,G_0}^\vee$ of the extension by zero map $j_{G_1,G_0} : C^\infty(G_1,k) \rightarrow C^\infty(G_0,k)$.
\end{itemize}
The Pontrjagin dual of $C^\infty(G_0,k)$ being $\Omega(G_0)$ we have, using \eqref{f:K-Kohlhaase}, the isomorphism
\begin{equation*}
   P_{G_0} : \Ext^d_{\Mod(\Omega(G_0))}(k,\Omega(G_0)) \xrightarrow{\cong} \Ext^d_{\Mod(G_0)}(C^\infty(G_0,k),k) \ .
\end{equation*}
Combining it with the above two diagrams we arrive at the commutative diagrams
\begin{equation}\label{diag:K11}
  \xymatrix{
    E^d(k) \ar[d]_{\chi_G^\mathrm{Koh}(g) \cdot} \ar[rrr]^-{P_{G_0} \circ \ell_{G,G_0}}_-{\cong} &&& \Ext^d_{\Mod(G_0)}(C^\infty(G_0,k),k) \ar[d]^{g_*} \\
    E^d(k) \ar[rrr]^-{P_{g G_0 g^{-1}}  \circ \ell_{G,g G_0 g^{-1}}}_-{\cong}  &&& \Ext^d_{\Mod(g G_0 g^{-1})}(C^\infty(g G_0 g^{-1},k),k)   }
\end{equation}
and
\begin{equation}\label{diag:K21}
  \xymatrix@R=0.5cm{
                &         \Ext^d_{\Mod(G_0)}(C^\infty(G_0,k),k) \ar[dr]^{\res}     \\
  E^d(k) \ar[ur]^-{P_{G_0} \circ \ell_{G,G_0}\qquad }_{\cong} \ar[dr]_-{P_{G_1} \circ \ell_{G,G_1}\qquad}^{\cong}  & &  \Ext^d_{\Mod(G_1)}(C^\infty(G_0,k),k)  \ar[dl]^{\qquad \Ext^d(j_{G_1,G_0},k)}            \\
                &         \Ext^d_{\Mod(G_1)}(C^\infty(G_1,k),k) .                }
\end{equation}
Specializing to $G_0 = U$ again we note that the duality isomorphism \eqref{f:duality-iso} for $V = C^\infty(U,k)$ and $i=0$ is given by
\begin{align*}
  \Ext^d_{\Mod(U)}(C^\infty(U,k),k) & \xrightarrow{\;\cong\;} \Hom_k(\Hom_{\Mod(U)}(k,C^\infty(U,k)), H^d(U,k))  \\
  e & \longmapsto \big[\phi \mapsto \phi^*(e)\big] \ .
\end{align*}
Let $\con_U : k \rightarrow C^\infty(U,k)$ denote the map which sends $1 \in k$ to the constant function with value $1$ on $U$. Then the above isomorphism is equivalent to the
isomorphism
\begin{align*}
  \Ext^d_{\Mod(U)}(C^\infty(U,k),k) & \xrightarrow{\;\cong\;} H^d(U,k)  \\
  e & \longmapsto \con_U^*(e) \ .
\end{align*}
The first isomorphism being natural in conjugation by $g \in G$ and this conjugation sending $\con_U$ to $\con_{gUg^{-1}}$ we see that we have the commutative diagram
\begin{align}\label{diag:K31}
\xymatrix{
   \Ext^d_{\Mod(U)}(C^\infty(U,k),k) \ar[d]_{g_*} \ar[rr]^-{\con_U^*} && H^d(U,k) \ar[d]^{g_*} \\
   \Ext^d_{\Mod(gUg^{-1})}(C^\infty(gUg^{-1},k),k) \ar[rr]^-{\con_{gUg^{-1}}^*} && H^d(gUg^{-1},k).   }
\end{align}
Furthermore, if $U' \subseteq U$ is any open subgroup, then we have the commutative diagram of duality pairings
\begin{equation*}
  \xymatrix{
    \Ext^d_{\Mod(U)}(C^\infty(U,k),k) \ar[d]_{\res}  & \times  & H^0(U,C^\infty(U,k)) \ar[r] & H^d(U,k)  \\
    \Ext^d_{\Mod(U')}(C^\infty(U,k),k) \ar[d]_{\Ext^d(j_{U',U},k)}  & \times  & H^0(U',C^\infty(U,k)) \ar[u]_{\cores} \ar[r] & H^d(U',k) \ar[u]_{\cores} \ar@{=}[d] \\
    \Ext^d_{\Mod(U')}(C^\infty(U',k),k)  & \times  & H^0(U',C^\infty(U',k)) \ar[u]_{H^0(U',j_{U',U})} \ar[r] & H^d(U',k) .  \\
  }
\end{equation*}
Here the top. resp.\ bottom, rectangle is commutative by the top rectangle in Lemma \ref{lemma:pairings}, resp.\ the functoriality of the Yoneda pairing. Note that the middle column maps $\con_{U'}$ to $\con_U$. Hence we obtain the commutative diagram
\begin{align*}
\xymatrix{
   \Ext^d_{\Mod(U)}(C^\infty(U,k),k) \ar[d]_{\Ext^d(j_{U',U},k) \circ \res} \ar[r]^-{\con_U^*} & H^d(U,k)  \\
   \Ext^d_{\Mod(U')}(C^\infty(U',k),k) \ar[r]^-{\con_{U'}^*} & H^d(U',k)  \ar[u]^{\cores}    .   }
\end{align*}
By combining it with the diagram \eqref{diag:K21} we deduce the commutative diagram
\begin{equation*}
  \xymatrix@R=0.5cm{
                &         H^d(U,k) \ar[dr]_{\cong}     \\
  E^d(k) \ar[ur]^{\con_U^* \circ P_{U} \circ \ell_{G,U}\qquad}_{\cong} \ar[dr]_{\con_{U'}^* \circ P_{U'} \circ \ell_{G,U'}\qquad}^{\cong} && k    \\
                &         H^d(U',k)  \ar[uu]^{\cores}  \ar[ur]^{\cong}      ,         }
\end{equation*}
where the right hand oblique arrows are our standard identifications. This means that the isomorphism $\con_U^* \circ P_{U} \circ \ell_{G,U} : E^d(k) \xrightarrow{\cong} k$ does not depend on the subgroup $U$. With this information we consider the commutative diagram
\begin{equation*}
  \xymatrix{
    E^d(k) \ar[d]_{\chi_G^\mathrm{Koh}(g) \cdot} \ar[rrrrr]^{\con_U^* \circ P_{U} \circ \ell_{G,U}}_{\cong} &&&&& H^d(U,k) \ar[d]_{g_*} \ar[r]^-{\cong} & k \ar[d]^{\chi_G(g) \cdot} \\
    E^d(k) \ar[rrrrr]^{\con_{gUg^{-1}}^* \circ P_{gUg^{-1}} \circ \ell_{G,gUg^{-1}}}_{\cong} &&&&& H^d(gUg^{-1},k) \ar[r]^-{\cong} & k   }
\end{equation*}
which arises by combining \eqref{diag:K11} and \eqref{diag:K31}. Since the horizontal arrows coincide we conclude that $\chi_G^\mathrm{Koh}(g) = \chi_G(g)$.
\end{proof}

\begin{lemma}
  Suppose that $\mathbf{G}$ is a connected reductive group over a finite extension $\mathfrak{F}$ of $\mathbb{Q}_p$; if $G$ is an open subgroup of $\mathbf{G}(\mathfrak{F})$ then $\chi_G = 1$.
\end{lemma}
\begin{proof}
The above Prop.\ \ref{prop:Kohlhaase} together with \cite{Koh} Cor.\ 5.2 show the assertion in the case $\mathfrak{F} = \mathbb{Q}_p$. In general let $\mathbf{G}'$ denote the Weil restriction of $\mathbf{G}$ to $\mathbb{Q}_p$. It is shown in \cite{Oes} App.\ 3 that $\mathbf{G}'$ again is a connected linear algebraic group with the property that $\mathbf{G}(\mathfrak{F}) = \mathbf{G}'(\mathbf{Q}_p)$ as $p$-adic Lie groups. Since our field extension is separable it follows from loc.\ cit.\ A.3.4 that with $\mathbf{G}$ also $\mathbf{G}'$ is reductive. This reduces the general case to the case $\mathfrak{F} = \mathbb{Q}_p$.
\end{proof}

\section{Derived smooth duality}\label{sec:derived-duality}

We begin by recalling some general nonsense about the adjunction between tensor product and Hom-functor which for three $k$-vector spaces $V_1$, $V_2$, and $V_3$ is given by the linear isomorphism
\begin{align}\label{f:basic-adj}
  \Hom_k(V_1 \otimes_k V_2, V_3) & \xrightarrow{\;\cong\;} \Hom_k(V_1, \Hom_k(V_2,V_3)) \\
                  A & \longmapsto \lambda_A(v_1)(v_2) := A (v_1 \otimes v_2) \ .   \nonumber
\end{align}
Suppose that all three vector spaces carry a left $G$-action. Then $\Hom_k(V_1 \otimes_k V_2, V_3)$ and $\Hom_k(V_1, \Hom_k(V_2,V_3))$ are equipped with the $G \times G \times G$-action defined by
\begin{equation*}
  {^{(g_1,g_2,g_3)} A}(v_1 \otimes v_2) := g_3 A(g_1^{-1} v_1 \otimes g_2^{-1} v_2)  \quad\text{and}\quad   {^{(g_1,g_2,g_3)} \lambda}(v_1)(v_2) := g_3 (\lambda(g_1^{-1} v_1)(g_2^{-1} v_2))  ,
\end{equation*}
respectively. The above adjunction is equivariant for these two actions. If we restrict to the diagonal $G$-action, then the above adjunction induces the adjunction isomorphism
\begin{equation*}
  \Hom_{k[G]}(V_1 \otimes_k V_2, V_3) \xrightarrow{\;\cong\;} \Hom_{k[G]}(V_1, \Hom_k(V_2,V_3)) \ .
\end{equation*}
If the $G$-action on the $V_i$ is smooth then this also can be written as an isomorphism
\begin{equation}\label{f:smooth-adj}
  \Hom_{\Mod(G)}(V_1 \otimes_k V_2, V_3) \cong \Hom_{\Mod(G)}(V_1, \underline{\Hom}(V_2,V_3)) \ .
\end{equation}

Let $D(G)$ denote the unbounded derived category of $\Mod(G)$. The tensor product functor
\begin{align*}
  \Mod(G) \times \Mod(G) & \longrightarrow \Mod(G) \\
             (V_1,V_2) & \longmapsto V_1 \otimes_k V_2 \ ,
\end{align*}
where the $G$-action on the tensor product is the diagonal one, is exact in both variables. Therefore it extends directly (i.e., without derivation)
to the functor
\begin{align*}
  D(G) \times D(G) & \longrightarrow D(G) \\
             (V_1^\bullet,V_2^\bullet) & \longmapsto \mathrm{tot}_\oplus (V_1^\bullet \otimes_k V_2^\bullet) \ ,
\end{align*}
which we usually denote simply by $V_1^\bullet \otimes_k V_2^\bullet$.\footnote{This uses the fact that for any two complexes of vector spaces one of which is acyclic their tensor product is acyclic as well.} On the other hand, since $\Mod(G)$ is a Grothendieck category, we have for any $V_0$ in $\Mod(G)$ the total derived functor
\begin{equation*}
  R\underline{\Hom}(V_0,-) : D(G)  \longrightarrow D(G)
\end{equation*}
such that $R^j\underline{\Hom}(V_0,V) = \underline{\Ext}^j(V_0,V)$ for any $V$ in $\Mod(G)$ and $j \geq 0$. We want to extend this to a bifunctor $D(G)^{op} \times D(G) \rightarrow D(G)$. First we recall that $\Mod(G)$ has arbitrary direct products (but which are not exact); we will denote these by $\prod^\infty$ to avoid confusion with the cartesian direct product. Hence, for any two complexes $V_1^\bullet$ and $V_2^\bullet$ in $\Mod(G)$ we may define the complex
\begin{equation*}
  \underline{\Hom}^\bullet (V_1^\bullet, V_2^\bullet) := {\prod_{j \in \mathbb{Z}}}^\infty \underline{\Hom}(V_1^j,V_2^{j+ \bullet})
\end{equation*}
in $\Mod(G)$ in the usual way. By construction we have that
\begin{align}\label{f:limit-formula}
  \underline{\Hom}^\bullet (V_1^\bullet, V_2^\bullet) & = \varinjlim_K \big( \prod_{j \in \mathbb{Z}} \underline{\Hom}(V_1^j,V_2^{j+ \bullet}) \big)^K = \varinjlim_K \prod_{j \in \mathbb{Z}} \underline{\Hom}(V_1^j,V_2^{j+ \bullet})^K \nonumber   \\
    & = \varinjlim_K \prod_{j \in \mathbb{Z}} \Hom_{\Mod(K)}(V_1^j,V_2^{j+ \bullet})   \\
    & = \varinjlim_K \Hom_{\Mod(K)}^\bullet (V_1^\bullet, V_2^\bullet)  \nonumber
\end{align}
is the inductive limit over all compact open subgroups $K \subseteq G$ of the usual Hom-complexes for the abelian categories $\Mod(K)$.

The adjunction \ref{f:smooth-adj} shows that the assumptions of \cite{KS} Thm.\ 14.4.8 are satisfied (with $\mathcal{P}_i = \mathcal{C}_i = \Mod(G)$, $G$ the tensor product functor, and $F_1 = F_2 = \underline{\Hom}$). Hence we obtain the following result.

\begin{proposition}\label{prop:total-inner-Hom}
  The total derived functor $R\underline{\Hom}(-,-) : D(G)^{op} \times D(G)  \longrightarrow D(G)$ exists and can be computed by $R\underline{\Hom}(V_1^\bullet,V_2^\bullet) = \underline{\Hom}^\bullet(V_1^\bullet,J^\bullet)$ where $V_2^\bullet \xrightarrow{\simeq} J^\bullet$ is a homotopically injective resolution. Moreover, there are the natural adjunctions
\begin{equation*}
  \Hom_{D(G)}(V_1^\bullet \otimes_k V_2^\bullet,V_3^\bullet) = \Hom_{D(G)}(V_1^\bullet, R\underline{\Hom}(V_2^\bullet,V_3^\bullet))
\end{equation*}
and
\begin{equation*}
  R\Hom_{\Mod(G)}(V_1^\bullet \otimes_k V_2^\bullet,V_3^\bullet) = R\Hom_{\Mod(G)}(V_1^\bullet, R\underline{\Hom}(V_2^\bullet,V_3^\bullet))
\end{equation*}
for any $V_i^\bullet$ in $D(G)$.
\end{proposition}

\begin{corollary}\label{sym-monoidal}
   $(D(G), \otimes_k, k, R\underline{\Hom})$ is a closed symmetric monoidal category.
\end{corollary}

For $V_2 = k$ viewed as complex concentrated in degree zero we, in particular, obtain the total derived duality functor
\begin{equation*}
  R\underline{\Hom}(-,k) : D(G)^{op} \longrightarrow D(G)
\end{equation*}
such that $R^j\underline{\Hom}(V,k) = S^j(V)$ for any $V$ in $\Mod(G)$ and any $j \geq 0$. In order to see in which way $k$ is a dualizing object for $\Mod(G)$ we have to introduce two finiteness conditions. First we observe that by Remark \ref{rem:global-dim} and \eqref{f:limit-formula} we may use \cite{Har} Prop.\ I.7.6 to conclude that the functor
\begin{equation*}
  R\underline{\Hom}(-,k) : D^b(G)^{op} \longrightarrow D^b(G)
\end{equation*}
(is way-out in both directions and) respects the bounded subcategories.

Next we recall that a representation $V$ in $\Mod(G)$ is called admissible if, for any open subgroup $K \subseteq G$, the vector space of $K$-fixed vectors $V^K$ is finite dimensional. In fact, it suffices to check the defining condition for a single compact open subgroup $K$ (apply the Nakayama lemma to the dual $\Omega(K)$-module $V^\vee$ or see \cite{Koh} Lemma 1.7). The full subcategory $\Mod_{adm}(G)$ of admissible representations in $\Mod(G)$ is  a Serre subcategory (cf.\ \cite{Em1} Prop.\ 2.2.13). Hence we have the strictly full triangulated subcategories $D^b_{adm}(G) \subseteq D^b(G)$ and $D_{adm}(G) \subseteq D(G)$ of those complexes whose cohomology representations are admissible.

\begin{lemma}\label{lemma:admissible}
  The derived duality functor $R\underline{\Hom}(-,k)$ respects both subcategories $D^b_{adm}(G)$ and $D_{adm}(G)$.
\end{lemma}
\begin{proof}
It is shown in \cite{Koh} Cor.\ 3.15 that for an admissible representation $V$ in $\Mod(G)$ the representations $S^j(V)$ are admissible as well. Hence for an admissible $V$ the complex $R\underline{\Hom}(V,k)$ lies in $D^b_{adm}(G)$. On the other hand we have observed already that our functor is way-out in both directions in the sense of \cite{Har} \S7. Therefore our assertion follows from loc.\ cit.\ Prop.\ I.7.3.
\end{proof}

Let $V^\bullet$ be any complex in $\Mod(G)$ and fix an injective resolution $k \xrightarrow{\simeq} \mathcal{J}^\bullet$. We construct a natural transformation
\begin{equation}\label{f:biduality-morphism}
  \eta_{V^\bullet} : V^\bullet \longrightarrow \underline{\Hom}^\bullet(\underline{\Hom}^\bullet(V^\bullet,\mathcal{J}^\bullet),\mathcal{J}^\bullet)
\end{equation}
as follows. Inserting the definitions we have to produce, for any $\ell \in \mathbb{Z}$, a natural $G$-equivariant map
\begin{equation*}
  \eta_{V^\ell} : V^\ell \longrightarrow {\prod_{j \in \mathbb{Z}}}^\infty \underline{\Hom}({\prod_{i \in \mathbb{Z}}}^\infty \underline{\Hom}(V^i,\mathcal{J}^{i+j}),\mathcal{J}^{j+\ell})
\end{equation*}
compatible with the differentials. It is straightforward to check that the maps $\eta_{V^\ell}(v)((f_{i,j})_i) := (-1)^{\ell j}f_{\ell,j}(v)$ have these properties.

\begin{proposition}\label{prop:involutive}
   If the complex $V^\bullet$ has admissible cohomology then the natural transformation $\eta_{V^\bullet}$ is a quasi-isomorphism.
\end{proposition}
\begin{proof}
Since we have a natural transformation between way-out functors the lemma on way-out-functors (\cite{Har} Prop.\ I.7.1(iii)) tells us that we need to establish the assertion only in the case where our complex is a single admissible representation (viewed as a complex concentrated in degree zero). In fact, by loc.\ cit.\ Prop.\ I.7.1(iv) we can go one step further. Suppose given a class $\mathcal{P}$ of admissible representations such that every admissible representation is embeddable into a finite direct sum of representations in this class. Then it suffices to check the assertion for representations in $\mathcal{P}$. We cannot apply this directly, though. First let us fix a compact open subgroup $K$ in $G$. Then we observe:
\begin{itemize}
  \item[--] Any admissible $G$-representation $V$ is also admissible as a $K$-representation;
  \item[--] $k \xrightarrow{\simeq} \mathcal{J}^\bullet$ is also an injective resolution in $\Mod(K)$;
  \item[--] the natural transformation $\eta_V$ remains the same if constructed for $V$ considered only as a $K$-representation.
\end{itemize}
This means that, for the purposes of our proof, we may assume that our group $G$ is compact. Let $C^\infty(G,k)$ denote, as before, the vector space of $k$-valued locally constant functions on $G$. Equipped with the left translation action it is an admissible smooth $G$-representation. We have $C^\infty(G,k)^\vee = \Omega(G)$. Let $V$ be any admissible representation in $\Mod(G)$. Then $V^\vee$ is a finitely generated (pseudocompact) $\Omega(G)$-module (\cite{Koh} Prop.\ 1.9(i)). Hence we find a surjection $\Omega(G)^m \twoheadrightarrow V^\vee$ in $\Mod_{pc}(G)$ for some integer $m \geq 0$. It is the dual of an injective map $V \hookrightarrow C^\infty(G,k)^m$ in $\Mod(G)$. Therefore we can take the single object $C^\infty(G,k)$ for the class $\mathcal{P}$. By \cite{Koh} Prop.\ 3.13 we have, for any integer $j$, that
\begin{equation*}
  R^j\underline{\Hom}(C^\infty(G,k),k) = S^j(C^\infty(G,k)) \cong
  \begin{cases}
  \chi_G \otimes_k C^\infty(G,k) & \text{for $j = d$},  \\
  0   & \text{otherwise},
  \end{cases}
\end{equation*}
where $\chi_G : G \rightarrow k^\times$ is Kohlhaase's duality character. Hence $R\underline{\Hom}(C^\infty(G,k),k) \simeq (\chi_G \otimes_k C^\infty(G,k))[-d]$ and then $R\underline{\Hom}(R\underline{\Hom}(C^\infty(G,k),k),k) \simeq C^\infty(G,k)$. One checks from the proof in loc.\ cit.\ that the latter quasi-isomorphism is induced by the natural transformation $\eta_{C^\infty(G,k)}$.
\end{proof}

\begin{corollary}\label{coro:involutive}
   On $D_{adm}(G)$ the functor $R\underline{\Hom}(-,k)$ is involutive.
\end{corollary}

\section{Globally admissible complexes}

In this section we will generalize some of the results in section \ref{sec:derived-duality} to a subcategory of $D(G)$ which is potentially larger than $D_{adm}(G)$. The possible drawback is that the defining condition for this subcategory is a ``global'' finiteness condition.

We let $\Vect$ denote the abelian category of $k$-vector spaces and $D(k)$ its unbounded derived category. In the following we fix an open subgroup $U \subseteq G$ which is pro-$p$ and torsion free. As recalled in Remark \ref{rem:global-dim} the functor
\begin{align*}
  \Mod(G) & \longrightarrow \Vect \\
  V & \longmapsto V^U = H^0(U,V)
\end{align*}
has finite cohomological dimension $d$. Hence its total derived functor $RH^0(U,-) : D(G) \longrightarrow D(k)$ exists (cf.\ \cite{Har} Cor.\ I.5.3)).

On the other hand the functor $\Hom_k(-,k)$ on $\Vect$ of taking the $k$-linear dual is exact and therefore passes directly to a functor form $D(k)^{op}$ to $D(k)$ which, for simplicity, we also denote by $\Hom_k(-,k)$.

\begin{proposition}\label{prop:diagram}
   The diagram
\begin{equation*}
  \xymatrix{
    D(G)^{op} \ar[d]_{\mathrm{forget}} \ar[rr]^-{R\underline{\Hom}(-,k)} && D(G) \ar[d]^{\mathrm{forget}} \\
    D(U)^{op} \ar[d]_{RH^0(U,-)} \ar[rr]^-{R\underline{\Hom}(-,k)} && D(U) \ar[d]^{RH^0(U,-)} \\
    D(k)^{op} \ar[rr]^-{\Hom_k(-,k)[-d]} && D(k)   }
\end{equation*}
is commutative (up to a natural isomorphism).
\end{proposition}
\begin{proof}
The upper rectangle is commutative since the forgetful functor $\Mod(G) \rightarrow \Mod(U)$, having the compact induction $\ind_U^G$ as an exact left adjoint, preserves injective as well as homotopically injective resolutions. For the lower triangle we first observe that the second adjunction formula in Prop.\ \ref{prop:total-inner-Hom} tells us that the composed functor $RH^0(U,R\underline{\Hom}(-,k))$ is naturally isomorphic to the functor $R\Hom_{\Mod(U)}(-,k)$. Hence it remains to exhibit a natural isomorphism between $R\Hom_{\Mod(U)}(-,k)$ and $\Hom_k(RH^0(U,-),k)[-d]$. For this we start with the Yoneda pairing
\begin{equation*}
  R\Hom_{\Mod(U)}(V^\bullet,k) \times R\Hom_{\Mod(U)}(k,V^\bullet) \longrightarrow R\Hom_{\Mod(U)}(k,k) \ .
\end{equation*}
By our assumption on the group $U$ the natural homomorphism $\sigma_{\leq d} R\Hom_{\Mod(U)}(k,k) \xrightarrow{\cong} R\Hom_{\Mod(U)}(k,k)$ is an isomorphism and the upper truncation $\sigma_{\leq d} R\Hom_{\Mod(U)}(k,k)$ at degree $d$ (cf.\ \cite{Har} p.\ 69/70) maps to its cohomology $H^d(U,k)[-d] \cong k[-d]$ in degree $d$. The Yoneda pairing therefore induces a pairing
\begin{equation*}
  R\Hom_{\Mod(U)}(V^\bullet,k) \times R\Hom_{\Mod(U)}(k,V^\bullet) \longrightarrow k[-d]
\end{equation*}
and hence a natural homomorphism
\begin{equation*}
  \Hom_k(R\Hom_{\Mod(U)}(k,V^\bullet),k[-d]) \longrightarrow R\Hom_{\Mod(U)}(V^\bullet,k) \ .
\end{equation*}
To show that it is an isomorphism we need to check that the map induced on cohomology
\begin{equation}\label{f:hyper-duality-iso}
  \Hom_k(H^{d-*}U,V^\bullet),k) \longrightarrow \Ext^*_{\Mod(U)}(V^\bullet,k)
\end{equation}
is bijective. If $V^\bullet$ is a single representation in degree zero then we have seen this already in \eqref{f:duality-iso}. By the Example 1 on p.\ 68 in \cite{Har} the functor $RH^0(U,-)$ and hence also the functor $\Hom_k(R\Hom_{\Mod(U)}(k,-),k[-d])$ are way-out in both directions. Similarly, by Remark \ref{rem:global-dim} and \cite{Har} Prop.\ I.7.6 the functor $R\Hom_{\Mod(U)}(-,k)$ is way-out in both directions as well. Hence it follows from \cite{Har} Prop.\ I.7.1(iii) that \eqref{f:hyper-duality-iso} always is bijective.
\end{proof}

\begin{definition}
   A complex $V^\bullet$ in $D(G)$ is globally admissible if its cohomology groups $H^i(U,V^\bullet)$, for any $i \in \mathbb{Z}$, are finite dimensional vector spaces. Let $D(G)^a \subseteq D(G)$ denote the strictly full triangulated subcategory of all globally admissible complexes.
\end{definition}

We will see only later in Cor.\ \ref{coro:independent} that this definition, indeed, does not depend on the choice of $U$.
To rephrase the definition let $D_{fin}(k) \subseteq D(k)$ denote the strictly full triangulated subcategory of all objects all of whose cohomology vector spaces are finite dimensional. Then $D(G)^a$ is the full preimage in $D(G)$ of $D_{fin}(k)$ under the functor $RH^0(U,-)$.

\begin{corollary}\label{coro:duality-a}
  The duality functor $R\underline{\Hom}(-k)$ respects the subcategory $D(G)^a$.
\end{corollary}
\begin{proof}
This is immediate from Prop.\ \ref{prop:diagram} since the functor $\Hom_k(-,k)$ on $D(k)$ respects the subcategory $D_{fin}(k)$.
\end{proof}

In \eqref{f:biduality-morphism} we had introduced the biduality morphism $\eta_{V^\bullet} : V^\bullet \rightarrow R\underline{\Hom}(R\underline{\Hom}(V^\bullet,k),k)$. Our further analysis of it will be based upon the following general observation.

\begin{lemma}\label{lemma:iso-criterion}
  A homomorphism $V_1^\bullet \rightarrow V_2^\bullet$ in $D(G)$ is an isomorphism if and only the induced map $H^i(U,V_1^\bullet) \rightarrow H^i(U,V_2^\bullet)$, for any $i \in \mathbb{Z}$, is bijective.
\end{lemma}
\begin{proof}
This is an immediate consequence of the equivalence $H$ between $D(G)$ and the derived category of a certain differential graded algebra in \cite{DGA} Thm.\ 9, which we will recall in section \ref{sec:dgHecke}. By construction the functor $H$ has the property that $h^*(H(-)) = H^*(U,-)$.
\end{proof}

\begin{proposition}\label{prop:involutive2}
   The biduality morphism $\eta_{V^\bullet}$, for any $V^\bullet$ in $D(G)$, is an isomorphism if and only if $V^\bullet$ lies in $D(G)^a$.
\end{proposition}
\begin{proof}
According to Lemma \ref{lemma:iso-criterion} we have to check that the maps
\begin{equation*}
  H^i(U,\eta_{V^\bullet}) : H^i(U,V^\bullet) \rightarrow H^i(U, R\underline{\Hom}(R\underline{\Hom}(V^\bullet,k),k))
\end{equation*}
are bijective for any $i \in \mathbb{Z}$ if and only if $V^\bullet$ lies in $D(G)^a$. By Prop.\ \ref{prop:diagram} we have natural isomorphisms
\begin{equation*}
  \xi^i_{V^\bullet} :  H^i(U,R\underline{\Hom}(V^\bullet,k)) \xrightarrow{\;\cong\;} \Hom_k(H^{d-i}(U,V^\bullet),k) \ .
\end{equation*}
We now claim that the diagram
\begin{equation*}
  \xymatrix{
    H^i(U,V^\bullet) \ar[d]_{b} \ar[rr]^-{H^i(U,\eta_{V^\bullet})} && H^i(U, R\underline{\Hom}(R\underline{\Hom}(V^\bullet,k),k)) \ar[d]^{\xi^i_{R\underline{\Hom}(V^\bullet,k)}}_{\cong} \\
    \Hom_k(\Hom_k(H^i(U,V^\bullet),k),k) \ar[rr]^-{\Hom_k(\xi^{d-i}_{V^\bullet},k)}_-{\cong} && \Hom_k(H^{d-i}(U,R\underline{\Hom}(V^\bullet,k)),k),   }
\end{equation*}
where $b$ denotes the natural map from a $k$-vector space into its double dual, is commutative up to the sign $(-1)^{i(d-i)}$. This immediately shows that $H^i(U,\eta_{V^\bullet})$ is bijective if and only if $b$ is bijective which, of course, is the case if and only if the vector space $H^i(U,V^\bullet)$ is finite dimensional.

To establish this claim we compute $R\underline{\Hom}(-,k)$ by using an injective resolution $J^\bullet$ of $k$ in $\Mod(G)$ and hence in $\Mod(U)$. Then $R\underline{\Hom}(V^\bullet,k) = \underline{\Hom}^\bullet(V^\bullet,J^\bullet)$ by Prop.\ \ref{prop:total-inner-Hom}. Moreover the adjunction property \eqref{f:smooth-adj} implies that $\underline{\Hom}^\bullet(V^\bullet,J^\bullet)$ always is homotopically injective. Finally we may also assume that $V^\bullet$ is homotopically injective. Our diagram therefore becomes
\begin{equation*}
  \xymatrix{
    h^i((V^\bullet)^U) \ar[d]_{b} \ar[rr]^-{H^i(U,\eta_{V^\bullet})} && \Hom_{K(U)}(\prod_{j \in \mathbb{Z}} \underline{\Hom}(V^j,J^{j + \bullet}),J^\bullet[i]) \ar[d]^{\xi^i_{R\underline{\Hom}(V^\bullet,k)}}_{\cong} \\
    \Hom_k(\Hom_k(h^i((V^\bullet)^U),k),k) \ar[rr]^-{\Hom_k(\xi^{d-i}_{V^\bullet},k)}_-{\cong} && \Hom_k(\Hom_{K(U)}(V^\bullet,J^\bullet[d-i]),k),   }
\end{equation*}
where $K(U)$ denotes as usual the homotopy category of unbounded complexes in $\Mod(U)$. We first recall that, once we fix an identification $h^d((J^\bullet)^U) = H^d(U,k) \cong k$, the map $\xi_{V^\bullet}^i$ is explicitly given by
\begin{align*}
  \xi_{V^\bullet}^i : \Hom_{K(U)}(V^\bullet,J^\bullet[i]) & \longrightarrow \Hom_k(h^{d-i}((V^\bullet)^U),k)  \\
                                       [\epsilon^\bullet] & \longmapsto \Big[ [\delta_{d-i}] \longmapsto [\epsilon^{d-i}(\delta_{d-i})] \Big] \ .
\end{align*}
Now let $[v_i] \in h^i((V^\bullet)^U)$. By definition of $\eta_{V^\bullet}$ its image under the top horizontal arrow in the above diagram is the homotopy class of the homomorphism of complexes
\begin{align*}
  \prod_{j \in \mathbb{Z}} \underline{\Hom}(V^j,J^{j + \bullet}) & \longrightarrow J^\bullet[i] \\
  (f_{j,\bullet})_\bullet & \longmapsto (-1)^{i \bullet} f_{i,\bullet}(v_i) \ .
\end{align*}
of degree $i$. Under the right perpendicular arrow it is further mapped to the linear map
\begin{align}\label{f:b}
  \Hom_{K(U)}(V^\bullet,J^\bullet[d-i]) & \longrightarrow k \\
  [(f_{j,d-i})_j] & \longmapsto (-1)^{i(d-i)} [f_{i,d-i}(v_i)] \ .   \nonumber
\end{align}
But $[(f_{j,d-i})_j]$ corresponds under $\xi^{d-i}_{V^\bullet}$ to the linear map in $\Hom_k(h^i((V^\bullet)^U),k)$ sending $[\delta_i]$ to $[f_{i,d-i}(\delta_i)]$. Hence the preimage of \eqref{f:b} under the bottom horizontal map in the diagram is equal to $(-1)^{i(d-i)} b([v_i])$ as claimed.
\end{proof}

\begin{corollary}\label{coro:independent}
  The subcategory $D(G)^a$ in $D(G)$ is independent of the choice of the subgroup $U \subseteq G$.
\end{corollary}

What is the relation between the subcategories $D_{adm}(G)$ and $D(G)^a$? We had observed earlier that a representation $V$ in $\Mod(G)$ is admissible if and only if the vector space $H^0(U,V)$ is finite dimensional. Moreover, by \cite{Em2} Lemma 3.3.4, we have the following fact.

\begin{lemma}\label{lemma:Emerton}
   If $V$ in $\Mod(G)$ is an admissible representation in $\Mod(G)$ then all the vector spaces $H^i(U,V)$, for $i \geq 0$, are finite dimensional.
\end{lemma}

This lemma says that, for an admissible $V$, the complex $RH^0(U,V)$ lies in $D_{fin}(k)$. By the Example 1 on p.\ 68 in \cite{Har} the functor $RH^0(U,-)$ is way-out in both directions. Therefore \cite{Har} Prop.\ I.7.3(iii) implies that the functor $RH^0(U,-)$ maps $D_{adm}(G)$ to $D_{fin}(k)$. This proves the following.

\begin{proposition}\label{respects-subcat}
  $D_{adm}(G) \subseteq D(G)^a$. 
\end{proposition}

Alternatively this can be seen by combining Prop.\ \ref{prop:involutive} and Prop.\ \ref{prop:involutive2}. On the subcategory $D^+(G)$ of bounded below complexes we have a stronger result.

\begin{proposition}\label{characterize-subcat}
   A complex $V^\bullet$ in $D^+(G)$ lies in $D_{adm}(G)$ if and only if $H^i(U,V^\bullet)$ is finite dimensional for any $i \in \mathbb{Z}$; i.e., we have $D^+(G) \cap D_{adm}(G) = D^+(G) \cap D(G)^a$.
\end{proposition}
\begin{proof}
The direct implication holds true by Prop.\ \ref{respects-subcat}. For the reverse implication we now assume that all the $H^i(U,V^\bullet)$ are finite dimensional.

Choose an integer $m$ such that $h^j(V^\bullet) = 0$ for any $j < m$. In this situation it is a standard fact (cf.\ \cite{KS} Exer.\ 13.3) that we have $H^0(U,h^m(V^\bullet)) = R^mH^0(U,V^\bullet) = H^m(U,V^\bullet)$. Hence $H^0(U,h^m(V^\bullet))$ is finite dimensional. As recalled before Lemma \ref{lemma:Emerton} this implies that $h^m(V^\bullet)$ is admissible. Moreover, Lemma \ref{lemma:Emerton} then says that $H^i(U,h^m(V^\bullet))$ is finite dimensional for any $i \in \mathbb{Z}$. We now use the distinguished triangles
\begin{equation*}
  \xymatrix{
       &  h^m(V^\bullet)[-m] \ar[dl]_{+1}             \\
 \tau^{\leq m-1} V^\bullet \ar[rr]  & & \tau^{\leq m} V^\bullet \ar[ul]       }
\ \text{and}\
  \xymatrix{
       & \tau^{\geq m+1} V^\bullet \ar[dl]_{+1}             \\
 \tau^{\leq m} V^\bullet \ar[rr]  & & V^\bullet \ar[ul]       }
\end{equation*}
in $D(G)$ (cf.\ \cite{KS} Prop.\ 13.1.15(i)). Since $\tau^{\leq m-1} V^\bullet \simeq 0$ in $D(G)$ the left triangle implies that $H^i(U, \tau^{\leq m} V^\bullet) \cong H^{i-m} (U, h^m(V^\bullet))$ is finite dimensional for any $i \in \mathbb{Z}$. Using this as an input for the long exact cohomology sequence associated with the right triangle we conclude that $H^i(U, \tau^{\geq m+1} V^\bullet)$ is finite dimensional for any $i \in \mathbb{Z}$ as well. But $h^j(\tau^{\geq m+1} V^\bullet) = 0$ for any $j \leq m$. Therefore we may repeat our initial reasoning for the complex $\tau^{\geq m+1} V^\bullet$ and deduce that $h^{m+1}(\tau^{\geq m+1} V^\bullet) = h^{m+1}(V^\bullet)$ is admissible. Proceeding inductively in this way we obtain that $h^j(V^\bullet)$ is admissible for any $j \in \mathbb{Z}$.
\end{proof}

\begin{lemma}\label{lemma:i-i}
   For any $V^\bullet$ in $D(G)$ and any $i \in \mathbb{Z}$ we have: If $H^i(U,V^\bullet) = 0$ then $h^i(V^\bullet) = 0$.
\end{lemma}
\begin{proof}
This is almost literally the same proof as the one for the reverse implication in \cite{DGA} Prop.\ 5.
\end{proof}

\begin{corollary}\label{coro:finitedim1}
  Any globally admissible complex $V^\bullet$ in $D(G)$ such that $H^i(U,V^\bullet) = 0$ for any sufficiently small $i$ lies in $D^+(G)$ and hence in $D_{adm}(G)$.
\end{corollary}

\begin{corollary}\label{coro:finitedim2}
   $D_{adm}^b(G)$ is the subcategory of all complexes $V^\bullet$ in $D(G)$ whose total cohomology $H^*(U,V^\bullet)$ is finite dimensional.
\end{corollary}
\begin{proof}
The direct implication is a consequence of Lemma \ref{lemma:Emerton} using the hypercohomology spectral sequence. The reverse implication follows from Prop.\ \ref{characterize-subcat} and Lemma \ref{lemma:i-i}.
\end{proof}

\begin{remark}\label{rem:compact-adm}
  If $G$ is compact then the natural functor $D^+(\Mod_{adm}(G)) \xrightarrow{\simeq} D_{adm}^+(G) := D^+(G) \cap D_{adm}(G)$ is an equivalence.
\end{remark}
\begin{proof}
This follows from \cite{Em2} Prop.\ 2.1.9 and \cite{Har} Prop.\ I.4.8.
\end{proof}

\bigskip

\noindent {\it{E-mail addresses}}: {\texttt{pschnei@wwu.de}, {\texttt{csorensen@ucsd.edu}}

\noindent {\sc{Peter Schneider, Math. Institut, Universit\"a{}t M\"{u{nster, M\"{u}nster, Germany.}}

\noindent {\sc{Claus Sorensen, Dept. of Mathematics, UC San Diego, La Jolla, USA.}}

\end{document}